\documentclass[a4paper,12pt]{amsart}

\usepackage{epsfig}
\usepackage{amsthm,amsfonts}
\usepackage{amssymb,graphicx,color}

\newtheorem{theorem}{Theorem}[section]
\newtheorem{lemma}[theorem]{Lemma}

\newtheorem{definition}[theorem]{Definition}
\newtheorem{example}[theorem]{Example}
\newtheorem{remark}[theorem]{Remark}

\newcommand{\R}{\mathbb R}
\newcommand{\Q}{\mathbb Q}

\newcommand{\F}{\mathbb F}

\begin{document}

\title[Multi-$\mathcal{K}$-Lipschitz Equivalence in dimension two.
 ]{Multi-$\mathcal{K}$-Lipschitz Equivalence in dimension two
}

\author[]{Lev Birbrair*}\thanks{*Research supported under CNPq 302655/2014-0 grant and by Capes-Cofecub}
\address{Lev Birbrair: Departamento de Matem\'atica, Universidade Federal do Cear\'a
(UFC), Campus do Pici, Bloco 914, Cep. 60455-760. Fortaleza-Ce, Brasil and \newline Institute of Mathematics, Jagiellonian University, Profesora Stanisława Łojasiewicza 6, 30-348 Kraków, Poland} \email{lev.birbrair@gmail.com}

\author[]{Rodrigo Mendes}
\address{Rodrigo Mendes: Instituto de ci\^encias exatas e da natureza, Universidade de Integra\c{c}\~ao Internacional da Lusofonia Afro-Brasileira (unilab), Campus dos Palmares, Cep. 62785-000. Acarape-Ce, Brasil and \newline Departament of Mathematics, Ben Gurion University of the Negev, P.O.B. 653, Be'er Sheva 84105, Israel.} 
\email{rodrigomendes@unilab.edu.br/mendespe@post.bgu.ac.il}


\begin{abstract}
In this paper, we study Multi-$\mathcal{K}$-equivalence of multi-germs of functions on the plane, definable in a polynomially bounded o-minimal structure.  As in \cite{birbrair2014lipschitz}, we partition the germ of the plane at origin into zones of arcs in such a way that it produces a non-Archimedean space (set of orders and width functions) compatible with a given multigerm, encoding its asymptotic behaviour. Such a partition is called Multipizza. We show the existence, uniqueness and complete invariance of Multipizzas with respect to the Multi-$\mathcal{K}$-Lipschitz equivalence.
\end{abstract}

\maketitle

\section{Introduction}

The paper can be considered as continuation the investigation initiated in \cite{K-Lip}, where the authors proved that the classification of the polynomial functions with respect to the  $\mathcal{K}$-Lipschitz equivalence is tame, i.e. has no moduli. The results from \cite{birbrair2014lipschitz} present a complete classification of the Lipschitz semialgebraic function germs in two-dimensional case. The next question is the classification of the map germs, not only the function germs. The finiteness theorem was proved by Valette and Ruas \cite{ruas&valette} (see also \cite{comEdvalter2}). The problem of multi-$\mathcal{K}$-Lipschitz equivalence is situated between the classification of function-germs and the map-germs.

The invariant constructed in \cite{birbrair2014lipschitz} is called Pizza. It is a finite partition of the germ of the plane into H\"older triangles, elementary with respect to the asymptotic behaviour of the function near the singular point. There the authors constructed a so-called width function, describing this behaviour. An elementary sub-triangle, or in other words pizza triangle is a triangle, such that the width-function is affine.  In this paper we define so-called Multipizza. As in \cite{birbrair2014lipschitz} in order to construct the pizza and the multipizza, we consider the Valette link of the plane, i.e. the set of germs of definable arcs, parameterized by the distance to the origin. The Valette link, equipped with the order of contact of arcs can be considered as a non-Archimedean metric space.
The Pizza of each function germ defines a partition of the Valette link of the plane. If we consider a multi-germ, we construct a multipizza as a minimal common refinement of a subset of partitions of the Valette link of the plane. We prove that the multipizza is a complete invariant with respect to the multi-$\mathcal{K}$-Lipschitz equivalence.

\section{Definitions and basic concepts}

All sets, functions and maps in this paper are assumed to be definable in a polynomially bounded
o-minimal structure over $\R$ with the field of exponents $\F$.
The simplest (and most important in applications) examples of such structures are real semialgebraic and globally subanalytic sets, with ${\F}={\Q}$.

An \emph{arc} in $\R^n$ is a germ at the origin of mapping $\gamma:[0,\epsilon)\rightarrow \R^n$ such that $\gamma(0)=0$. Unless otherwise specified, we suppose that arcs are parametrized by the distance to the origin, i.e., $\|\gamma(t)\|=t$.

Given two subsets $X,Y \subset \R^n$, a map $F\colon X\rightarrow Y$ is called Lipschitz when there exists a constant $C\ge 1$ such that $\|F(x_1)-F(x_2)\|\leq C\|x_1-x_2\|, \ \forall x_1, x_2 \in X$.
If the map $F$ is bijective and $\frac{1}{C}\|x_1-x_2\|\leq \|F(x_1)-F(x_2)\|\leq C\|x_1-x_2\|$ for some $C\ge 1$, we say that $F$ is a bi-Lipschitz map and that $X$ and $Y$ are bi-Lipschitz equivalent.
\begin{definition}\label{ordonarc}
\emph{Let $f\not\equiv 0$ be (a germ at the origin of) a function defined on an arc $\gamma$.
The order $\alpha$ of $f$ on $\gamma$ (notation $\alpha=ord_\gamma f$ or  $\alpha=ord_t f(\gamma(t))$) is the value $\alpha\in\F$ such that $f(\gamma(t))=c t^{\alpha}+o(t^{\alpha})$ as $t\to 0$, where $c\ne 0$.
If $f\equiv 0$ on $\gamma$, we set $ord_\gamma f=\infty$.}
\end{definition}

\begin{definition}\label{tord}
\emph{ The \emph{order of tangency} of two arcs $\gamma$ and $\gamma'$ (notation $tord(\gamma,\gamma')$)
is defined as $ord_t\|\gamma'(t)-\gamma(t)\|$.
The arcs $\gamma$ and $\gamma'$ are \emph{tangent} if $tord(\gamma,\gamma')>1$.
 }
\end{definition}

\begin{definition}\label{holder}
\emph{Two arcs $\gamma_1\ne\gamma_2$ in $\R^2$ divide the germ of $\R^2$ at the origin into two components. If $\beta=tord(\gamma_1,\gamma_2)>1$ then the closure of the smaller (not containing a half-plane) component is called a \emph{$\beta$-H\"older triangle}. If $tord(\gamma_1,\gamma_2)=1$ then the closure of any of the two components is called a \emph{$1$-H\"older triangle}. The number $\beta \in \F_+$ is called the \emph{exponent} of the H\"older triangle. We denote by $T(\gamma_1,\gamma_2)$ a H\"older triangle bounded by $\gamma_1$ and $\gamma_2$. The arcs $\gamma_1$ and $\gamma_2$ are called its \emph{boundary arcs}. Any germ $(U,0)$ of a definable set $U\subset\R^n$ bi-Lipschitz equivalent to a $\beta$-H\"older triangle $T\subset\R^2$ is also called a $\beta$-H\"older triangle.}
\end{definition}

\section{Zones and pizza decomposition}
\begin{definition}\label{zone}
\emph{Let $Z$ be a set of arcs. We say that $Z$ {\em contains} a H\"older triangle $T$ (notation $T\sqsubset Z$) if any arc $\gamma\subset T$ belongs to $Z$.
A set of arcs $Z$ is called a \emph{zone} if for any two arcs $\gamma_1$ and $\gamma_2$ of $Z$ there exists a H\"older triangle $T=T(\gamma_1,\gamma_2)\sqsubset Z$.}
\end{definition}

\begin{example}\label{zones}

\emph{(i) The set $\{\gamma\}$ consisting of a single arc $\gamma$ is a zone. Such a zone is called \emph{singular}.}

\emph{(ii) For a H\"older triangle $T$, the set $Z_T$ of all arcs $\gamma\subset T$ is a zone.}

\emph{(iii) For a H\"older triangle $T$, the set of all arcs belonging to the interior of $T$ is a zone.}

\emph{(iv) For a $\beta$-H\"older triangle $T$, the set $Z^\circ_T$ of all arcs $\gamma\subset T$ having the order of tangency $\beta$
with both boundary arcs of $T$ is a zone.}

\emph{(v) For a $\beta$-H\"older triangle $T=T(\gamma_1,\gamma_2)$, the set of all arcs $\gamma\subset T$ having the order of tangency greater than $\beta$ with its boundary arc $\gamma_1$ is a zone.}

\end{example}

\begin{definition}
\emph{The \emph{order} $\mu(Z)$ of a zone $Z$ is the infimum of $tord(\gamma_1,\gamma_2)$, where the arcs $\gamma_1$ and $\gamma_2$ belong to $Z$.
If $Z$ is a singular zone then $\mu(Z)=\infty$.}
\end{definition}

\begin{definition}\label{QZf}
\emph{For a zone $Z$ and a continuous function $f$ defined on each $\gamma \in Z$, let $Q_{Z,f}$ be the set of values of $ord_{\gamma} f$ for all arcs $\gamma \in Z$. Then $Q_{Z,f}$ is either a point or a segment (open, closed or semi-open) in $\F_+ \cup \infty$. When $Q_{Z,f}=\{q\}$ is a point, we call $Z$ a $q$-zone for $f$, or simply $q$-zone when $f$ is clear from the context. A $q$-zone is \emph{maximal} if it is not a proper subset of a larger $q$-zone for $f$.
}
\end{definition}

\begin{definition}\emph{
A zone $Z$ is \emph{closed} if there are two arcs $\gamma_1$ and $\gamma_2$ in $Z$ such that $tord(\gamma_1,\gamma_2)=\mu(Z)$.}
\end{definition}

\begin{remark}
\emph{In example \ref{zones}, the zones in (i)-(iv) are closed while the zone in (v) is not.}
\end{remark}

\begin{definition}\label{Zgamma}
\emph{
Let $f$ be a continuous function defined on a H\"older triangle $T$. If $\gamma$ is an arc in $T$ and $q=ord_\gamma f$, the maximal $q$-zone for $f$ containing $\gamma$ is denoted by $Z_{\gamma}f$. It follows from \cite[Remark 3.4]{birbrair2014lipschitz} that $Z_{\gamma}f$ is a closed zone.}
\end{definition}

\begin{definition}\label{generic}
\emph{An arc $\gamma$ is called \emph{generic} with respect to a (closed) zone $Z$ if there exists a $\beta$-H\"older triangle $T\sqsubset Z$,
where $\beta=\mu(Z)$, such that $\gamma\in Z^\circ_T$. A closed zone $Z$ is called \emph{perfect} if any arc in $Z$ is generic.}
\end{definition}

\begin{remark}
\emph{In example \ref{zones}, the zones in (i) and (iv) are perfect while the zones in (ii) and (iii) are not.}
\end{remark}

\begin{definition}\emph{
A zone $Z$ is called a \emph{monotonicity zone} for $f$ if the set of arcs $\gamma\subset Z$ such that $ord_\gamma f=q$ is a zone for any $q\in Q_{Z,f}$.}
\end{definition}

\begin{remark} \emph{A H\"older triangle $T$ is called \emph{elementary} with respect to a function $f$ defined on $T$ (see \cite[Definition 2.7]{birbrair2014lipschitz}) if, for any two distinct arcs $\gamma_1$ and $\gamma_2$ in $T$ such that $\rm{ord}_{\gamma_1}(f)=\rm{ord}_{\gamma_2}(f)=q$,
the order of $f$ is equal to $q$ on any arc in the H\"older triangle $T(\gamma_1,\gamma_2)\subset T$.
The zone $Z_T$ of all arcs $\gamma\subset T$ is a monotonicity zone for $f$ if and only if $T$ is elementary with respect to $f$. }
\end{remark}

\begin{definition}\label{def:mu}\emph{
If $Z$ is a monotonicity zone of $f$ then for any $q\in Q_{Z,f}$ there exists a unique zone $Z_q\subset Z$ which is a maximal $q$-zone for $f$.
We define the \emph{width function} $\mu:Q_{Z,f} \rightarrow \F_+\cup\{\infty\}$ as $\mu(q)=\mu(Z_q)$.}
\end{definition}
\begin{remark}
\emph{ Let $Z$ be a monotoniciy zone for a function $f$. For any $\gamma \in Z$ such that $q=ord_\gamma f$, one has $Z_q=Z_\gamma f$.}
\end{remark}

\begin{remark}\label{rmk:mu}
\emph{Let $T$ be an elementary $\beta$-H\"older triangle with respect to a Lipschitz function $f$.
Then the width function $\mu(q)$ on $Q=Q_{Z_T,f}$ associated with the zone $Z_T$ of all arcs in $T$ has the following properties:
$\mu(q)$ is piecewise-affine, either non-decreasing or non-increasing but not necessarily continuous, on $Q$ (see \cite{DS} or \cite[Theorem 3.2]{birbrair2014lipschitz});
$\mu(q)\le q$ for all $q\in Q$; $\min_q \mu(q)=\beta$; $\mu(q)$ is constant only when $Q$ is a point.
}
\end{remark}

\begin{lemma} (Minimal Pizza decomposition)\label{MP}
\emph{Let $f:(T,0) \rightarrow (\R,0)$ be a Lipschitz function. Then there exists a finite family of perfect zones $Z_i$ of arcs in $T$ satisfying the following:}

\begin{enumerate}
  \item Two of the zones $Z_i$ are the boundary arcs of $T$.
  \item $Z_i \cap Z_j=\emptyset$ for $i\ne j$.
  \item Any family of arcs $\{\gamma_i\}$ such that $\gamma_i \in Z_i$ decomposes $T$ into H\"older triangles elementary with respect to $f$, such that these triangles and the corresponding width functions present a minimal pizza for $f$ on $T$.
  \item Any minimal pizza for the function $f$ can be obtained in this way, i.e., taking arcs $\gamma_i\in Z_i$.
\end{enumerate}
\end{lemma}

\begin{proof}
   Consider a partition of $T$ into H\"older triangles $T_i=T(\gamma_i,\gamma_{i+1})$ with a minimal pizza $\mathcal{H}=\{\beta_i,Q_i,s_i,\mu_i\}$ for $f$ constructed in \cite{birbrair2014lipschitz}.
    Here $\beta_i$ is the exponent of $T_i$, $Q_i\subset\F_+\cup\{\infty\}$ is the set (either a point or a segment) of values $ord_\gamma f$ for all $\gamma\subset T_i$, $s_i\in\{+,-,0\}$ is the sign of $f$ on the interior of $T_i$, and $\mu_i:Q_i\to\F_+\cup\{\infty\}$ (the width function) is an affine function (a constant if $Q_i$ is a point). Let $\gamma_i$ be the common boundary arc of triangles $T_{i-1}$ and $T_i$. If $f|_{\gamma_i}\equiv 0$ then $Z_i=\{\gamma_i\}$ is a singular zone, thus it is perfect. Assume that $f|_{\gamma_i}\not\equiv 0$ and consider the following cases:

  \medskip

  (i) If $\mu_{i-1}(\gamma_i)=\mu_i(\gamma_i)=\mu$, we define $Z_i=Z_{\gamma_i} f$.
  We claim that any arc $\gamma\in Z_i$ belongs to the interior of the H\"older triangle $T_{i-1}\cup T_i$.
  This is obvious when the exponents $\beta_{i-1}$ and $\beta_i$ of $T_{i-1}$ and $T_i$ are both less than $\mu$
  (these exponents cannot be greater than $\mu$).
  If both $\mu_{i-1}$ and $\mu_i$ are not constant, each of the two triangles contains an arc $\gamma$ with $\mu(ord_\gamma f)\ne\mu$, thus $Z_i$
  cannot extend beyond $\gamma$.
  It remains to consider the case when one of the two triangles, say $T_i$, has exponent $\mu$ and $\mu(ord_\gamma f)=\mu$ for each arc $\gamma\subset T_i$.
  In such a case, our pizza $\mathcal{H}$ would be a refinement of a smaller pizza where $T_{i-1}\cup T_i$ is replaced by a single
  H\"older triangle, in contradiction to the assumption that it is a minimal pizza.

  Thus any arc $\gamma\in Z_i=Z_{\gamma_i} f$ belongs to the interior of $T_{i-1}\cup T_i$.
  Now we prove that any $\gamma\in Z_i$ is generic.
  Let $\gamma\in Z_i$ be an arc in $T_i$ (the case of $\gamma$ being an arc in $T_{i-1}$ is similar).

If $\beta_i<\mu$ then $\mu$ is the maximal value of the affine function $\mu_i$. Thus, since $\gamma_i$ is necessarily the supporting arc of $T_i$, any arc $\gamma'\subset T_i$ such that $tord(\gamma',\gamma_i)\ge\mu$
belongs to $Z_i$. We can choose $\gamma'$ in $T_i$ so that $\gamma\subset T(\gamma_i,\gamma')$ and $tord(\gamma,\gamma')=\mu$. Observe that since $tord(\gamma,\gamma_i)\geq \mu$, one has $tord(\gamma',\gamma_i)\geq \mu$ (i.e., $\gamma' \in Z_i$). Let $\gamma''\in Z_i$ be an arc in $T_{i-1}$ such that $tord(\gamma_i,\gamma'')=\mu$ (such an arc always exists by the definition of the width function).
 Then $\gamma\subset T(\gamma',\gamma'')$ and $tord(\gamma,\gamma')=tord(\gamma,\gamma'')=\mu$, thus $\gamma$ is generic.

If $\beta_i=\mu$ then $\mu$ is the minimal value of $\mu_i$, and $tord(\gamma,\gamma_{i+1})=\mu$ ($\mu_i$ is not constant). Hence, in this case, the supporting arc of $T_i$ is $\gamma_{i+1}$. It gives that any arc $\gamma'$ in $T_i$ such that $tord(\gamma',\gamma_{i+1})=\mu$ belongs to $Z_i$.
Thus we can choose $\gamma'\in Z_i$ so that $\gamma\subset T(\gamma_i,\gamma')$ and $tord(\gamma,\gamma')=\mu$.
Choosing $\gamma''\subset T_{i-1}$ as before, we prove that $\gamma\subset T(\gamma',\gamma'')$ is generic.

  (ii) If $\mu_{i-1}(\gamma_i)<\mu_i(\gamma_i)=\mu$, we define $Z_i$ as the union of the two zones of order $\mu$:
  $Z'=\{\gamma\subset T_i: ord_\gamma f=ord_{\gamma_i} f\}$ and $Z''=\{\gamma\subset T_{i-1}: tord(\gamma,\gamma_i)\ge\mu\}$.
  As in case (i), we show first that any arc $\gamma\in Z_i$ belongs to the interior of $T_i\cup T_{i-1}$.
  This is obvious for $\gamma\in Z''$ since $\mu_{i-1}(\gamma_i)<\mu$.
  For $\gamma\in Z'$, this is obvious when $\beta_i<\mu$. If $\beta_i=\mu$, one has that $\mu$ is the minimal value of $\mu_i$ (with $\mu_i$ non constant). Thus $tord(\gamma,\gamma_{i+1})=\mu$ for any arc $\gamma\in Z'$ in this case.


  Now we can prove that any arc $\gamma\in Z_i$ is generic.
  If $\gamma\in Z'$ and $\beta_i<\mu$ then $\mu$ is the maximal value of the affine function $\mu_i$, and $Z'$ is the set of all arcs $\gamma'\subset T_i$ such that
  $tord(\gamma',\gamma_i)=\mu$. Notice, one can choose an arc $\gamma'$ such that $\gamma \subset T(\gamma',\gamma_i)$ and $tord(\gamma',\gamma)=\mu$. Then, $\gamma'\in Z_i$. Also, we can always choose such an arc $\gamma''\in Z''$
  such that $tord(\gamma'',\gamma_i)=\mu$. Then $\gamma\subset T(\gamma',\gamma'')$ satisfies $tord(\gamma,\gamma')=tord(\gamma,\gamma'')=\mu$. Thus $\gamma$ is generic.
  If $\gamma\in Z''$, since any arc $\gamma''\subset T_{i-1}$ such that $tord(\gamma'',\gamma_i)=\mu$ belongs to $Z''$, we can choose $\gamma''$ so that
  $\gamma\subset T(\gamma'',\gamma_i)$ and $tord(\gamma,\gamma'')=\mu$. For any arc $\gamma'\in Z'$ such that $tord(\gamma',\gamma_i)=\mu$ we have
  $\gamma\subset T(\gamma',\gamma'')$ and $tord(\gamma,\gamma')=tord(\gamma,\gamma'')=\mu$. Thus $\gamma$ is generic.

The case $\mu_{i-1}(\gamma_i)>\mu_i(\gamma_i)=\mu$ is similar. Thus all zones $Z_i$ are perfect.
To prove that $Z_i\cap Z_j=\emptyset$ for $i\ne j$, it is enough to consider the case $j=i+1$ and show that
$Z_i$ and $Z_{i+1}$ do not have common arcs in $T_i$.



\vspace{0.5cm}

Observe that the affine function $\mu_i$ has distinct values on the boundary arcs of $T_i$. We may assume that $\mu_i(\gamma_i)>\mu_i(\gamma_{i+1})=\beta_i$. Then $tord(\gamma,\gamma_i)\ge\mu_i(\gamma_i)$ for any arc $\gamma\in Z_i$, and $tord(\gamma',\gamma_i)=\beta_i$ for any arc $\gamma'\in Z_{i+1}$. Thus, the equality $\gamma=\gamma'$ is impossible.

This completes the proof of Lemma \ref{MP}.
\end{proof}

\section{Multipizza}

\begin{definition}\label{multi0}
\emph{Let $\mathbf{f}=\{f_1,\ldots,f_n\}$ be a set of germs of Lipschitz functions $f_\nu:(T,0)\to (\R,0)$ defined on an oriented H\"older triangle $T$.
A \emph{geometric multipizza} $\mathcal{M}$ associated to $\mathbf f$  is a 
decomposition of $T$ into H\"older triangles $T_i$
ordered according to the orientation of $T$, such that $T_i\cap T_j=\{0\}$ when $|i-j|>1$, and the intersection of $T_i$ and $T_{i+1}$ is their common boundary arc $\gamma_i$, with the following properties:
\begin{enumerate}
\item Each H\"older triangle $T_i$ is elementary with respect to each function $f_\nu$;
\item Each width function $\mu_{i\nu}:Q_{i\nu}\to \F_+\cup\{\infty\}$ for $f_\nu|_{T_i}$, where $Q_{i\nu}=Q_{Z_{T_i},f_\nu}$, is affine or constant if $Q_{i\nu}$ is a point;
\item For each $T_i$ (unless $Q_{i\nu}$ is a point for each $\nu$) one of its two boundary arcs (called the \emph{supporting arc} $\tilde\gamma_i$ of $T_i$)
  is selected so that, for each arc $\gamma \subset T_i$ and each $\nu$ such that $ord_\gamma f_\nu \neq ord_{\tilde\gamma_i} (f_\nu)$, we have $\mu_{i\nu}(q)=tord(\gamma,\tilde\gamma_i)$;
\item The sign of each function $f_\nu$ on the interior of each H\"older triangle $T_i$ is $s_{i\nu}\in\{+,-,0\}$ .
\end{enumerate}}
\end{definition}

\begin{remark}\label{compatibility}
\emph{Let $q_{i\nu}$ be the endpoint of a segment $Q_{i\nu}$ where $\mu_{i\nu}$ has its maximum, and let $q'_{i\nu}$ be its other endpoint (if $Q_{i\nu}$ is a point then $q_{i\nu}=q'_{i\nu}$).
Then the order of $f_\nu$ on the supporting arc $\tilde\gamma_i$ of $T_i$ is $q_{i\nu}$.
The continuity of $f_\nu$ on $T$ implies the following \emph{compatibility conditions} of the data in Definition \ref{multi0} for each common boundary arc $\gamma_i$ of the H\"older triangles $T_i$ and $T_{i+1}$:}

\begin{equation}\label{eq:compatibility}
\begin{split}
 q_{i\nu}=q_{i+1,\nu} \text{ when } \tilde\gamma_i=\gamma_i=\tilde\gamma_{i+1},\quad  q'_{i\nu}=q_{i+1,\nu} \text{ when }  \tilde\gamma_i\ne\gamma_i=\tilde\gamma_{i+1},\\
 q_{i\nu}=q'_{i+1,\nu}  \text{ when }   \tilde\gamma_i=\gamma_i\ne\tilde\gamma_{i+1}, \quad
 q'_{i\nu}=q'_{i+1,\nu}  \text{ when } \tilde\gamma_i\ne\gamma_i\ne\tilde\gamma_{i+1}.
 \end{split}
 \end{equation}
 \end{remark}

\begin{definition}\label{multi1}
\emph{Two geometric multipizzas $\mathcal M$ and $\mathcal M'$ are \emph{combinatorially equivalent} if
\begin{enumerate}
  \item[i.] The two 
  decompositions $\{T_i\}$ and $\{T'_i\}$ of $T$ into H\"older triangles have the same sequence $\{\beta_i\}$ of exponents sorted on the same way;
  \item[ii.] Selections of supporting arcs $\tilde\gamma_i$ and $\tilde\gamma'_i$ for triangles $T_i$ and $T'_i$ of the two 
  decompositions are compatible, i.e., the sequence of supporting arcs are sorted on the same way, preserving the compatibility conditions given in (\ref{eq:compatibility});
  \item[iii.] $Q_{i,\nu}=Q'_{i,\nu}$ and $\mu_{i,\nu}=\mu'_{i,\nu}$ for all $i$ and $\nu$;
  \item[iv.] The signs $s_{i\nu}$ and $s'_{i\nu}$ of the functions $f_\nu$ in the interiors of $T_i$ and $T'_i$ are the same for all $i$ and $\nu$.
\end{enumerate}}
 \end{definition}

\begin{definition}\label{abstractholder}\emph{
An \emph{abstract H\"older complex} on an oriented $\beta$-H\"older triangle $T$ is an ordered sequence $\{\beta_i\}$
of exponents $\beta_i\in\F_+$, with $\min_i \beta_i=\beta$. This is equivalent to a bi-Lipschitz equivalence class of decompositions $\{T_i\}$ of $T$
into $\beta_i$-H\"older triangles $T_i$, ordered according to the orientation of $T$.}
\end{definition}

\begin{definition}\label{abs}
\emph{An \emph{abstract multipizza} $\mathcal{H}$ is an abstract H\"older complex $\{\beta_i\}$ on an oriented H\"older triangle $T$,
and for each $i$ a finite collection
$\mathcal{H}_i=\big\{\{Q_{i\nu}\},\{\mu_{i\nu}\},\{s_{i\nu}\}\big\},$
where
\begin{enumerate}
\item Each $Q_{i\nu}=[a_{i\nu},b_{i\nu}]$ is a closed directed segment of $\F_+\cup\{\infty\}$, where ``directed'' means either $a_{i\nu}<b_{i\nu}$ or $a_{i\nu}>b_{i\nu}$ (or else $a_{i\nu}=b_{i\nu}$ and $Q_{i\nu}$ is a point) with the compatibility conditions $b_{i\nu}=a_{i+1,\nu}$;
\item $\mu_{i,\nu}:Q_{i,\nu}\to \F_+\cup\{\infty\}$ is an affine function, non-constant unless $Q_{i\nu}$ is a point, such that $\mu_{i\nu}(q)\le q$ and $\min_q \mu_{i\nu}(q)=\beta_i$;
\item For each $i$, unless each $Q_{i\nu}$ is a point, either $\mu_{i\nu}(a_{i\nu})=\beta_i$ for all $\nu$, or $\mu_{i\nu}(b_{i\nu})=\beta_i$ for all $\nu$;
\item Each $s_{i\nu}$ is an element of $\{+,-,0\}$.
\end{enumerate}
}\end{definition}

\begin{remark}\label{abstract-geometric}
\emph{We associate an abstract multipizza $\mathcal H$ with a geometric multipizza $\mathcal M$ by writing each segment $Q_{i\nu}$ of $\mathcal M$ as a directed segment $[a_{i\nu},b_{i\nu}]$, when $Q_{i\nu}$ is not a point. This is done by setting $a_{i\nu}=ord_{\gamma_{i-1}} f_\nu$ and
$b_{i\nu}=ord_{\gamma_i} f_\nu$, where $\gamma_i$ is the common boundary arc of $T_i$ and $T_{i+1}$.
One can easily check that two geometric multipizzas are combinatorially equivalent exactly when the same abstract multipizzas are associated with them.}
\end{remark}
\begin{example}\label{multipizza.example}
\emph{Let $f_1,f_2,f_3\colon (T=\R^2_{\geq 0},0)\rightarrow (\R,0)$ be the function germs given by $f_1(x,y)=x^2+y^4, \ f_2(x,y)=x$ e $f_3(x,y)=x^3-y^2$. One get the geometric multipizza of the multigerm $f=\{f_1,f_2,f_3\}$ as follows:
}\begin{itemize}
       \item \emph{The triangle $T_{11}=T=T_{21}$ is elementary with respect the functions $f_1$ and $f_2$. Its correspondent segments are $Q_{11}=[2,4]$ and $Q_{21}=[1,+\infty]$. }

        \item \emph{The arc $\tilde\gamma_{12}=\{(0,t)\}_{t\geq 0}=\tilde\gamma_{22}$ is the arc supporting of $T$, providing the width functions $\mu_{11}(q)=\frac{q}{2}$ and $\mu_{21}(q)=q$ for functions $f_1$ and $f_2$ respectively.}

  \item \emph{An elementary decomposition of $T$ with respect to $f_3$ is $T_{31}\cup T_{32}\cup T_{33}=T(\gamma_{31},\gamma_{32})\cup T(\gamma_{32},\gamma_{33})\cup T(\gamma_{33},\gamma_{34})$, where $\gamma_{31}(t)=(t,0), \gamma_{32}(t)=(t,t^{\frac{3}{2}}), \gamma_{33}(t)=(t,2t^{\frac{3}{2}})$ and $\gamma_{34}(t)=(0,t)$.}
        \item \emph{On these triangles, we have $Q_{31}=[3,\infty], \ Q_{32}=[\infty,3], \ Q_{33}=[3,2]$ respectively.}

        \item \emph{The supporting arc of $T_{31}$ and $T_{32}$ is  $\gamma_{32}(t)=(t,t^{\frac{3}{2}})$; the supporting arc of $T_{33}$ is $\gamma_{33}(t)=(t,2t^{\frac{3}{2}})$;}

        \item \emph{The width functions are $\mu_{31}(q)=q-\frac{3}{2}=\mu_{32}(q)$ and $\mu_{33}(q)=\frac{q}{2}$.}
    \end{itemize}
    \emph{Then, the H\"older complex associated to the multipizza $\mathcal{M}$ is the H\"older complex associated to the pizza of the function $f_3$. It corresponds to a refinement of the pizzas of $f_1$ and $f_2$. It provides the multipizza of $\mathcal{M}$ of $f$.}

\end{example}
\begin{definition}\label{multi2}\emph{A geometric multipizza $\mathcal{M}'$ is called a \emph{refinement} of a multipizza $\mathcal{M}$ if
\begin{enumerate}
  \item The decomposition $\{T'_j\}$ of the multipizza $\mathcal{M}'$ is a proper refinement of the partition $\{T_i\}$ of the multipizza $\mathcal{M}$;
  \item For each $T'_j\subset T_i$ and for each $\nu$, we have $Q'_{j,\nu}\subset Q_{i,\nu}$ and $\mu'_{j,\nu}=\mu_{i,\nu}|_{Q'_{j,\nu}}$.
\end{enumerate}
A multipizza is called \emph{minimal} if there exists no multipizza $\overline{\mathcal{M}}$ such that $\mathcal{M}$ is a refinement of $\overline{\mathcal{M}}$.}
\end{definition}

 \begin{remark}
 \begin{enumerate}
 \item[i.] \emph{For $n=1$ and $\mathbf{f}=\{f\}$, a multipizza is the same as a pizza defined in \cite{birbrair2014lipschitz}, and Definitions \ref{multi0} - \ref{multi1}
 are the same as Definitions 2.12 - 2.13 in \cite{birbrair2014lipschitz}. Note that selection of supporting arcs $\tilde\gamma_i$ corresponds to orientation of $Q_i$ in \cite{birbrair2014lipschitz}.}
 \item[ii.] \emph{In the example \ref{multipizza.example}, the multipizza $\mathcal{M}$ is minimal.}
 \end{enumerate}
 \end{remark}

\begin{theorem}\label{Main}
Let $T\subset (\R^2,0)$ be a H\"older triangle, and $\mathbf{f}=(f_1,\,f_2,\ldots,f_n)$ a multigerm formed by Lipschitz functions $f_\nu:T \rightarrow (\R,0)$.
Then there exists a multipizza $\mathcal M$ of $\mathbf f$ satisfying Definition \ref{multi0}.

\begin{proof}
  We show the result for two functions and then we use the induction by the numbers of the functions. Let $f_1, f_2:(T,0)\rightarrow (\mathbb{R},0)$ be two Lipschitz definable functions. Let $\{Z_{\beta_i\}}$ be the family of perfect zones, constructed in the lemma \ref{MP} for $f_1$ and $\{\tilde Z_{\beta_j}\}$ be the family of perfect zones, constructed for $f_2$. Consider the family of zones $\{\tilde Z_S\}$ defined as follows:
  \begin{enumerate}
    \item If for some $i$, $Z_{\beta_i}$ does not intersect any $ \tilde Z_{\beta_j}$, then $Z_{\beta_i} \in \{\tilde Z_S\}$;
    \item If for some $j$, $ \tilde Z_{\beta_j}$ does not intersect to any $Z_{\beta_i}$, then $ \tilde Z_{\beta_j} \in \{\tilde Z_S\}$;
    \item If for some pair $i,j$, $Z_{\beta_i}\cap \tilde Z_{\beta_j}\neq \emptyset$, then $Z_{\beta_i}\cap \tilde Z_{\beta_j} \in \{\tilde Z_S\}$.
  \end{enumerate}
  Notice that a intersection of perfect zones is a perfect zone with order equal to $\beta_{ij}=max\{\beta_i,\beta_j\}$. Using the procedure, we obtain a family of perfect zones $\{\tilde Z_S\}$ such that for $S_1 \neq S_2$, then $\tilde Z_{S_1} \cap \tilde Z_{S_2}=\emptyset$ . Let us choose a family of arcs $\{\beta'_S\}$, such that for any $\tilde Z_S$ there exists a unique definable arc $\beta'_S$ such that $\beta'_S \subset \tilde Z_S$

  \medskip
  
The family of triangle $T_S=T(\beta'_{S-1},\beta'_S)$ defines a decomposition of the triangle $T$ into H\"older triangles. Notice that all the triangles, obtained by the procedure are elementar with respect to $f_1$ and $f_2$. It is because each triangle $T_S$ is a subtriangle of a given elementar triangle. And each width function corresponding to each $f_1$ and $f_2$ on each triangle $T_S$ is affine.  But these triangles does not satisfies necessarily the conditions of the theorem \ref{Main} because for some $\gamma \subset T_S$ may be happens that $\mu_1(\gamma)$ can be equal to $tord(\gamma, \tilde \beta_{S-1})$ and $\mu_2(\gamma)$ can be equal to $tord(\gamma, \tilde \beta_S)$. That is why we must improve the family of the arcs in order to obtain the condition of the theorem. Let $\tilde Z_{S-1,S}$ be a zone of the arcs $\gamma \subset T_S$ such that $tord(\gamma, \tilde \beta_S)=tord(\gamma, \tilde \beta_{S-1})$.

  \medskip

  Notice that the zone $\tilde Z_{S-1,S}$ is a perfect zone and it does not depend on the choice of the arcs $\tilde \beta_{S-1} \in \tilde Z_{S-1}$, $\tilde \beta_S \in \tilde Z_S$. Let us choose an arc $\tilde \beta_{S-1,S} \in \tilde Z_{S-1,S}$. Let $T_1$ and $T_2$ be the triangles obtained using the decomposition of $T_S$ by arc $\tilde \beta_{S-1,S}$. If the width $\mu_1(\gamma)$ for $f_1$ is counted as $tord(\gamma, \tilde \beta_{S-1})$ then for all $\gamma \subset T_2$ we get $tord(\gamma,\tilde \beta_{S-1})=tord(\tilde \beta_{S-1},\tilde \beta_S)$ and thus $ord_{\tilde \gamma} f_1$ is constant for all $\tilde \gamma \subset T_2$. By the same reason, we get $ord_{\tilde {\tilde \gamma}} f_2$ constant, for all $\tilde {\tilde \gamma} \subset T_1$. Thus, the triangles obtained by the arcs $\{\tilde \beta_S\}$ and $\{\tilde \beta_{S-1,S}\}$ satisfies the condition of the theorem \ref{Main}. We proved the statement for two functions.

  \medskip

  Induction on the number of functions:

  \medskip

  Suppose that for any $n-1$ functions, we can construct a family of zones $\{Z_i\}$ such that for any choice of the arcs $\gamma_i \in Z_i$, we obtain a decomposition, satisfying the Multipizza theorem. Consider a finite set of functions $f_1,f_2,\ldots,f_n$. Applying the induction hypothesis, we may obtain a family of perfect zones $\{Z_i\}$ such that the choice of the arcs $\gamma_i \in Z_i$ define a decomposition of $T$ at $0$ into H\"older triangles satisfying the Multipizza theorem for $f_1,f_2,\ldots,f_{n-1}$. Next, we apply the lemma \ref{MP} to the function $f_n$. We obtain a family of zones $\{\tilde Z_j\}$ such that any choice of $\tilde \gamma_j \in \tilde Z_j$ defines a pizza for $f_n$. Using the procedure described for a pair of functions $\{f_1,f_2\}$ for all the pairs $\{f_n,f_i\}$, i.e., for the family of zones $\{\tilde Z_j\}$, $\{Z_i\}$, $i=1,2,\ldots,n-1$, we obtain the result for $f_1,f_2,\ldots,f_n$.
\end{proof}
\end{theorem}
\begin{remark}
\emph{Notice that there is an essential difference between the interaction of pairs of Holder triangles and the interaction of pairs of perfect zones. From the non-Archimedean condition, we saw that if $\mu(\beta)=\mu(Z_1)=\mu(Z_2)$ and $Z_1\cap Z_2\neq \emptyset$ then $\mu(Z_1\cap Z_2)=\beta$. In other words, the order of intersection of perfect zones does not depend of its ``relative position". It is a subtle ingredient that allows the Multipizza algorithm works.}
\end{remark}

\section{Multi-$\mathcal{K}$-Lipschitz Equivalence}
\begin{definition} \emph{Let $f{=}\{f_1,f_2,\ldots,f_n\}{:}(\mathbb{R}^2,0)\rightarrow (\mathbb{R},0)$ and $g{=}\{g_1,\ldots,g_n\}{:}(\mathbb{R}^2,0)\rightarrow (\mathbb{R},0)$ be two functions multigerms. The multigerms $f$ and $g$ are called {\emph multi-$\mathcal{K}$-Lipschitz equivalent} if there exists a bi-Lipschitz homeomorphism $H{:}(\mathbb{R}^2\times \mathbb{R}^n,0)\rightarrow (\mathbb{R}^2\times \mathbb{R}^n,0)$,
\[
 \ H((x,y),y_1,y_2,\ldots,y_n)=(h(x,y),\tilde H_1((x,y),y_1),\ldots, \tilde H_n((x,y),y_n))
\]
such that $h:(\mathbb{R}^2,0)\rightarrow (\mathbb{R}^2,0)$ is a bi-Lipschitz homeomorphism and
$\tilde H_j((x,y),0)=0$, $j=1,2,\ldots,n$ and}
\begin{center}
  $H((x,y),f_1(x,y),f_2(x,y),\ldots,f_n(x,y))=(h(x,y),g_1(h(x,y)),\ldots,g_n(h(x,y)))$.
\end{center}
\end{definition}
\begin{remark}
\emph{Notice that, since
\[
H((x,y),0,\ldots,0,f_j(x,y),0\ldots,0){=}(h(x,y),0,\ldots,0,g_j(h(x,y)),0,\ldots,0),
\]
for all $j$, the functions $f_j$ and $g_j$ are $\mathcal{K}$-Lipschitz equivalent with the same bi-Lipschitz change of coordinates. As in \cite{comEdvalter2}, $h$ is also called the common factor of the multigerm of bi-Lipschitz homeomorphisms $H_j=(h,\tilde H_j)$, where $H_j\circ(Id_{(x,y)},f)=(h,g\circ h)$.}
\end{remark}
\begin{theorem}
\emph{Let $f{=}\{f_1,f_2,\ldots,f_n\}{:}(\mathbb{R}^2,0)\rightarrow (\mathbb{R},0)$ and $g{=}\{g_1,\ldots,g_n\}{:}(\mathbb{R}^2,0)\rightarrow (\mathbb{R},0)$ be two Lipschitz functions multigerms. If $f$ and $g$ have two combinatorially equivalent multipizzas then the germs $f$ and $g$ are multi-$\mathcal{K}$-Lipschitz equivalent.}
\end{theorem}
\begin{proof}
Let $\{T_i\}$ be a decomposition of ($\mathbb{R}^2,0)$ into H\"older triangles, corresponding to a multipizza of $f$ and $\{\tilde T_j\}$ be a decomposition of $(\mathbb{R}^2,0)$ into H\"older triangles, corresponding to a equivalent multipizza of $g$. Since, they are combinatorially equivalent, there exists a bi-Lipschitz automorphism $h:(\mathbb{R}^2,0)\rightarrow (\mathbb{R}^2,0)$ such that $h(\tilde T_i)= T_i$, for all $i$. Since the multipizza of $f$ is a pizza for each $f_j\circ h$, we conclude that, for each $j$, there are constants $C_1, C_2$ and $i\in \{0,1\}$ such that
\begin{equation}\label{same.contact}
  C_1g_j(x,y)\leq (-1)^i f_j(h(x,y)) \leq C_2g_j(x,y).
\end{equation}

We take the map $H\colon (\R^2\times\R^n,0)\rightarrow (\R^2\times \R^n,0)$ in a natural way by $H(x,y,y_1,\ldots,y_n)=(h(x,y),\tilde H_1(x,y,y_1),\ldots,\tilde H_n(x,y,y_n))$. It is easy to see that the map $H$ is a bi-Lipschitz map sending, for each $j$, $Graph(f_j)$ into $Graph(g_j)$ with bi-Lipschitz common factor $h$. Hence, the multigerms $f$ and $g$ are multi-$\mathcal{K}$-bi-Lipschitz equivalent.
\end{proof}
\begin{example}\label{example.multipizza2}
\emph{Let $f,g\colon (\R_{\ge 0}^2,0)\rightarrow (\R,0)$ be two bi-germs given by $f=\{f_1=x^2-y^3,f_2=y\}$ and $g=\{g_1=x^3-y^2,g_2=y\}$. Clearly, the germs $f_j$ and $g_j$ are $\mathcal{K}$-Lipschitz equivalent, $j=1,2$. Then, the pizzas of $f_j$ and $g_j$ are combinatorially equivalent, $j=1,2$. It gives that $f_j$ and $g_j$ are $\mathcal{K}$-Lipschitz equivalent. The multipizza of $f$ is}
\[ 
\mathcal{M}_f=\{\{T(\gamma_1,\gamma_2)), T(\gamma_2,\gamma_3), T(\gamma_3,\gamma_4)\},\{[3,\infty],[\infty,3],[3,2],[1,\infty]\},\{q-\frac{3}{2},2q,q\},\{-,+\}\} ,
\]
where $\gamma_1(t)=(0,t), \gamma_2(t)=(t^{\frac{3}{2}},t),\gamma_3(t)=(2t^{\frac{3}{2}},t)$ and $\gamma_4(t)=(0,t)$.

\emph{The multipizza of $g$ is 
\[
\mathcal{M}_g=\{\{T(\tilde\gamma_1,\tilde\gamma_2)), T(\tilde\gamma_2,\tilde\gamma_3), T(\tilde\gamma_3,\tilde\gamma_4)\}\{[3,\infty],[\infty,3],[3,2],[\infty,1]\},\{q-\frac{3}{2},2q,q\},\{+,-\}\},
\]
where $\tilde\gamma_1(t)=(t,0), \tilde\gamma_2(t)=(t,t^{\frac{3}{2}}), \tilde\gamma_3(t)=(t,2t^{\frac{3}{2}})$ and $\tilde\gamma_4(t)=(0,t)$.
Observe that $\mathcal{M}_f$ and $\mathcal{M}_g$ has the same segments (up to a reverse orientation at last segment) and same width functions. However, $\mathcal{M}_f$ and $\mathcal{M}_g$ are not combinatorially equivalent because the supporting arcs are not compatible. Then, the multigerms $f$ and $g$ cannot be Multi-$\mathcal{K}$-Lipschitz equivalent (see the theorem bellow).}
\end{example}

\begin{theorem}
\begin{enumerate}
  \item \emph{Let $f=\{f_1,\ldots,f_n\}$ be a multigerm. Then any two minimal multipizzas of $f$ are combinatorially equivalent.
  \item Let $f,g:(\mathbb{R}^2,0)\rightarrow (\mathbb{R},0)$ be two multigerms. Then they are multi-$\mathcal{K}$-bi-Lipschitz equivalent if and only if $f$ and $g$ have combinatorially equvalent minimal multipizzas.}
\end{enumerate}

\end{theorem}
\begin{proof}
Consider the minimal pizzas of the functions $f_1,\ldots,f_n$ and apply the algorithm described in the proof of Multipizza theorem (Theorem \ref{Main}). Notice, that the zones, described in the proof are perfect zones and that is why that the class of the Multipizzas, with respect to the combinatorial equivalence, does not depend on the choice of the arcs on those zones. Any resulting Multipizza is obtained from these zones and it is minimal because the pizzas of  $f_1,\ldots,f_n$ are minimal.

\vspace{0.25cm}

Suppose that the multigerms $f$ and $g$ are multi-$\mathcal{K}$-Lipschitz equivalent by a bi-Lipschitz map $H=(h,H_1,\ldots,H_n)$. Then the bi-Lipschitz map $h$ preserves all the exponents of all the triangles of the corresponding decompositions and from the pairs given by $(h,\tilde H_j)$, the widths of all arcs are preserved. That is why it preserves the minimal multipizza.
\end{proof}

\begin{example} \emph{It holds to mention that the Multi-$\mathcal{K}$-bi-Lipschitz equivalence is a more finer equivalence than $\mathcal{K}$-bi-Lipschitz equivalence. For instance, the multigerms $f,g\colon \mathbb{R}\rightarrow \mathbb{R}$ given by $f=\{x^2,x^3\}$ and $g=\{x^2,x^5\}$ are $\mathcal{K}$-bi-Lipschitz equivalent. In fact, we can construct a bi-Lipschitz homeomorphism $H\colon \mathbb{R}^3\rightarrow \mathbb{R}^3$ sending $Graph(f)$ into $Graph(g)$ preserving its tangent line. However, it is easy to see that $f$ and $g$ are not multi-$\mathcal{K}$-bi-Lipschitz equivalent.}\end{example}


\end{document}